\documentclass[reqno,11pt]{amsart}
\usepackage{amsmath,amssymb}
\newtheorem{theorem}{Theorem}[section]

\newtheorem{lemma}[theorem]{Lemma}
\newtheorem{proposition}[theorem]{Proposition}
\theoremstyle{definition}
\newtheorem{definition}[theorem]{Definition}
\theoremstyle{remark}

\numberwithin{equation}{section}

\title[Gibbs Fields: Uniqueness and Decay of Correlations]{Gibbs Fields: Uniqueness and Decay of Correlations. Revisiting Dobrushin and Pechersky}

\thanks{This work was financially supported by the DFG through SFB 701: "Spektrale Struk-
turen und Topologische Methoden in der Mathematik" and by the
European Commission under the project STREVCOMS PIRSES-2013-612669.
D. Conache also thanks the support of the IRTG (IGK) 1132
``Stochastics and Real World Models", Universit\"at Bielefeld.}

\author{ Diana Conache}

\address{Fakut\"at f\"ur Mathematik, Universit\"at Bielefeld, Bielefeld D-33615, Germany}
\email{dputan@math.uni-bielefeld.de}

\author{Yuri  Kondratiev}
\address{Fakut\"at f\"ur Mathematik, Universit\"at Bielefeld, Bielefeld D-33615, Germany}
\email{kondrat@math.uni-bielefeld.de}

\author{ Yuri  Kozitsky}

\address{Instytut Matematyki, Uniwersytet Marii Curie-Sk{\l}odowskiej, 20-031 Lublin, Poland}
\email{jkozi@hektor.umcs.lublin.pl}

\author{ Tanja Pasurek}

\address{Fakut\"at f\"ur Mathematik, Universit\"at Bielefeld, Bielefeld D-33615, Germany}
\email{pasurek@math.uni-bielefeld.de}

\begin{document}

\begin{abstract}
We give a detailed and refined proof of the Dobrushin-Pechersky
uniqueness criterion extended to the case of Gibbs fields on
general graphs and single-spin spaces, which in particular need not
be locally compact. The exponential decay of correlations under the
uniqueness condition has also been established.
\end{abstract}

\maketitle

\section{Introduction}

A random field on a countable set ${\sf L}$ is a collection of
random variables -- {\it spins}, indexed by $\ell \in {\sf L}$. Each
variable is defined on some probability space and takes values in
the corresponding {\it single-spin} space $\Xi_\ell$. Typically, it
is assumed that each $\Xi_\ell$ is a copy of a Polish space $\Xi$.
In a `canonical version', the probability space is $(\Xi^{\sf L},
\mathcal{B}(\Xi^{\sf L}), \mu)$, where $\mu$ is a probability
measure on the Borel $\sigma$-field $\mathcal{B}(\Xi^{\sf L})$. Then
also $\mu$ is referred to as random field. A particular case of such
a field is the product measure of some single-spin probability
measures $\sigma_\ell$. {\it Gibbs random fields} with pair
interactions are constructed as `perturbations' of the product
measure $\otimes_{\ell \in {\sf L}} \sigma_\ell$ by the `densities'
\[
 \exp\left(\sum W_{\ell \ell'} (\xi_\ell, \xi_{\ell'} ) \right)
\]
where $W_{\ell \ell'}: \Xi\times \Xi \to \mathbb{R}$ are measurable
functions -- {\it interaction potentials}, whereas the sum is taken
over the set ${\sf E} \subset {\sf L}\times {\sf L}$ such that
$W_{\ell\ell'}\neq 0$ for $(\ell,\ell')\in {\sf E}$. The latter
condition defines the underlying graph ${\sf G} = ({\sf L}, {\sf
E})$. For bounded potentials, the perturbed measures usually exist.
Moreover, there is only one such measure if the potentials are small
enough and the underlying graph is enough `regular'. If the
potentials are unbounded, both the existence and uniqueness issues
turn into serious problems of the theory. Starting from the first
results in constructing Gibbs fields with `unbounded spins'
\cite{Lebow}, attempts to elaborating tools for proving their
uniqueness were being undertaken \cite{COPP,DobP,Mal}. However,
except for the results of \cite{Mal} obtained for the potentials and
single-spin measures of a special type, and also for methods
applicable to `attractive' potentials, see \cite{KP,Pasurek,Diana},
there is only one work presenting a kind of general approach to this
problem. This work is due to R. L. Dobrushin and E. A. Pechersky
\cite{DobP}, which was first published in Russian and later on
translated to English. In spite of its great importance, the work
remains almost unknown (it has been cited only few times) presumably
for the following reasons: (i) the English translation in
\cite{DobP} was made with numerous typos and errors of mathematical
nature, whereas the Russian version is inaccessible for the most of
the readers; (ii) most of the proofs in \cite{DobP} are very
involved and complex, and essential parts of them are only sketched
or even missing. In the present publication, we give a refined and
complete proof of the Dobrushin-Pechersky result extended in the
following directions: (a) we do not employ the compactness arguments
crucially used in \cite{DobP};  (b) we settle (in Proposition
\ref{TVpn} below) the measurability issues not even discussed in
\cite{DobP}; (c) instead of the cubic lattice $\mathbb{Z}^d$ we
consider general graphs as underlying sets of the Gibbs fields. The
refinement consists, among others, in explicitly calculating the
threshold value of $K$ in (\ref{K}) and the constants in the basic
estimates in Lemma \ref{R3lm}. Due to (a), as the single-spin spaces
$\Xi$ one can consider just standard Borel spaces, e.g., infinite
dimensional spaces which are not locally compact, see
\cite{KP,Pasurek}. Due to (c), one can apply the criterion to varios
models employing graphs as underlying sets. One can also apply the
criterion to the equilibrium states of continuum particle systems,
see \cite[Chapter 4]{Diana} and Section \ref{222} below.

The structure of this paper is as follows. In Section \ref{2SEC}, we
give necessary preliminaries and formulate the results in Theorems \ref{1tm} and \ref{2tm}. Section
\ref{3SEC} contains the proof of these theorems
based on the estimates obtained in Lemmas \ref{R3lm} and \ref{dclm},
respectively, as well as on a number of other facts proved thein.
In Section \ref{4SEC}, we perform detailed constructions
yielding the proof of the mentioned lemmas.

\section{Setup and the Result}

\label{2SEC}

\subsection{Notations and preliminaries}
The underlying set for the spin configurations of our model is a
countable simple connected graph $({\sf L}, {\sf E})$.  For a vertex
$\ell \in {\sf L}$, by $\partial \ell$ we denote the neighborhood of
$\ell$, i.e., the  set of vertices adjacent to $\ell$. The vertex
{\it degree} $\varDelta_\ell$ is then the cardinality of $\partial
\ell$. The only assumption regarding the graph is that
\begin{equation}
  \label{1}
 \sup_{\ell\in {\sf L}}\varDelta_\ell=: \varDelta < \infty,
\end{equation}
i.e., the vertex degree is globally bounded. A given ${\sf V}\subset {\sf L}$ is said to be an {\it independent set} of vertices if
\begin{equation}
\label{2}
\forall \ell \in {\sf V} \quad \partial l \cap {\sf V} = \emptyset.
\end{equation}
The {\it chromatic number} $\chi\in \mathbb{N}$ is the smallest number such that
\begin{equation}
 \label{3}
 {\sf L} = \bigcup_{j=0}^{\chi-1} {\sf V}_j, \qquad {\sf V}_j \ - \ {\rm independent}, \ j=0, \dots , \chi-1.
\end{equation}
Obviously, $\chi \leq \varDelta+1$. However, by Brook's theorem,
see, e.g., \cite{Lovasz}, for our graph we have that $\chi \leq
\varDelta$.

For a measuarble space $(E, \mathcal{E})$, by $\mathcal{P}(E)$ we
denote the set of all probability measures on $\mathcal{E}$. All
measurable spaces we deal with in this article are standard Borel
spaces. The prototype example is a Polish space endowed with the
corresponding Borel $\sigma$-field. For $\sigma \in \mathcal{P}(E)$
and a suitable function $f:E \to \mathbb{R}$, we write
\begin{equation*}
 \sigma(f) = \int_{E} f d \sigma.
\end{equation*}
For our model, the single-spin spaces $(\Xi_\ell,
\mathcal{B}(\Xi_\ell))$, $\ell \in {\sf L}$, are copies of a
standard Borel space $(\Xi, \mathcal{B}(\Xi))$. Then the
configuration space $X = \Xi^{\sf L}$ equipped with the product
$\sigma$-field $\mathcal{B}(X)=\mathcal{B}(\Xi^{\sf L})$ is also a
standard Borel space. Likewise,  for a nonempty ${\sf D} \subset
{\sf L}$, $\Xi^{\sf D}$ is the product of $\Xi_\ell$, $\ell \in {\sf
D}$. Its elements are denoted by $x_{\sf D} = (x_\ell)_{\ell \in
{\sf D}}$, whereas the elements of $X$ are written simply as $x=
(x_\ell)_{\ell \in {\sf L}}$. For $y,z\in X$, by $y_{\sf D} \times
z_{{\sf D}^c}$ we denote the configuration $x\in X$ such that
$x_{\sf D} = y_{\sf D}$ and $x_{{\sf D}^c} = z_{{\sf D}^c}$, ${\sf
D}^c:= {\sf L} \setminus {\sf D}$. For ${\sf D}\subsetneq {\sf L}$,
we denote $\mathcal{F}_{\sf D} = \mathcal{B}(\Xi^{{\sf D}^c})$ and
write $\mathcal{F}_\ell$ if ${\sf D}=\{\ell\}$.
\begin{definition}
  \label{0df}
Given $\ell \in {\sf L}$, let $ \pi_\ell:=\{\pi_\ell^x : x \in
X\}\subset \mathcal{P}(\Xi_\ell)$ be such that the map $X\ni x \mapsto
\pi^x_\ell (A)\in \mathbb{R}$ is $\mathcal{F}_\ell$-measurable for
each $A\in \mathcal{B}(\Xi_\ell)$. A family $\pi=\{\pi_\ell\}_{\ell \in
{\sf L}}$ of the maps of this kind is said to be a {\it one-site}
specification.
\end{definition}
\begin{definition}
 \label{1df}
A given $\mu \in \mathcal{P}(X)$ is said to be {\it consistent} with
a one-site specification $\pi$ in a given ${\sf D}\subseteq {\sf L}$
if $\mu(\cdot|\mathcal{F}_\ell)(x) = \pi_\ell^x$ for $\mu$-almost
all $x$ and each $\ell \in {\sf D}$. By $\mathcal{M}_{\sf D}(\pi)$
we denote the set of all  $\mu\in \mathcal{P}(X)$ consistent with
$\pi$ in ${\sf D}$. We say that $\mu$ is consistent with $\pi$ if it
is consistent in ${\sf L}$, and write just $\mathcal{M}(\pi)$ in
this case.
\end{definition}
Obviously, $\mu\in \mathcal{M}(\pi)$ if and only if it satisfies the
following equation
\begin{eqnarray}
  \label{5}
\mu(A)& = & \int_X \int_X \mathbb{I}_A (x) \pi_\ell^y (dx_\ell) \prod_{\ell'\neq \ell} \delta_{y_{\ell'}} (d x_{\ell'}) \mu (dy)\\[.2cm]
& = & \int_X \left(\int_\Xi \mathbb{I}_A (\xi \times y_{\{\ell\}^c}) \pi_\ell^{y} (d\xi) \right)  \mu (dy),\nonumber
\end{eqnarray}
which holds for every $\ell\in {\sf L}$ and $A\in \mathcal{B}(X)$.
Here, for $\eta\in \Xi$, $\delta_\eta\in \mathcal{P}(\Xi)$ is the
Dirac measure centered at $\eta$ and $\mathbb{I}_A$ stands for the
indicator of $A$.

For a standard Borel space $(E, \mathcal{E})$, let $(E^2,
\mathcal{E}^2)$ be the product space. For $\sigma, \varsigma \in
\mathcal{P}(E)$, let $\varrho\in \mathcal{P}(E^2)$ be such that
$\varrho(A\times E)= \sigma(A)$ and $\varrho(E\times A) = \varsigma
(A)$ for all $A\in \mathcal{B}(E)$. Then we say that $\varrho$ is a
{\it coupling} of $\sigma$ and $\varsigma$. By $\mathcal{C}(\sigma,
\varsigma)$ we denote the set of all such couplings.

For $\xi, \eta \in \Xi$, we set
\begin{equation*}
 \upsilon(\xi, \eta) = \left\{ \begin{array}{ll} 0, \quad {\rm if} \ \  \xi = \eta;\\[.3cm]
                               1, \quad {\rm otherwise},
                              \end{array} \right.
\end{equation*}
which is a measurable function on $\Xi^2$ since $\Xi$ is a standard Borel space. Then we equip
$\mathcal{P}(\Xi)$ with the {\it total variation distance}
\begin{equation}
  \label{TV}
d(\sigma, \varsigma) = \sup_{A \in \mathcal{B}(\Xi)} |\sigma (A) -
\varsigma (A)|,
\end{equation}
that, by duality, can also be written in the form
\begin{equation*}
d(\sigma, \varsigma) = \inf_{\varrho\in \mathcal{C}(\sigma, \varsigma)} \int_{\Xi^2}  \upsilon (\xi, \eta) \varrho( d \xi , d \eta).
\end{equation*}
\begin{proposition}
  \label{TVpn}
For each $\ell \in {\sf L}$ and $(x,y)\in X^2$, there exists
$\varrho^{x,y}_\ell \in \mathcal{C}(\pi_\ell^x, \pi_\ell^y)$ such
that: (a) for each $B\in \mathcal{B}(\Xi^2_\ell)$, the map $X^2\ni
(x,y)\mapsto \varrho^{x,y}_\ell(B)$ is $\mathcal{F}_\ell^2$-measurable; (b) the following holds
\begin{equation}
 \label{7}
 d(\pi_\ell^x,  \pi_\ell^y) = \int_{\Xi^2}  \upsilon (\xi, \eta) \varrho_\ell^{x,y}( d \xi , d \eta).
\end{equation}
\end{proposition}
\begin{proof}
Set
\begin{equation*}
(\pi_\ell^x \wedge \pi_\ell^y) (A) = \min\{\pi_\ell^x (A);
\pi_\ell^y(A) \}, \qquad A\in \mathcal{B}(\Xi_\ell).
\end{equation*}
In view of the measurability as in Definition \ref{0df}, the map
$X^2\ni (x,y) \mapsto (\pi_\ell^x \wedge \pi_\ell^y) (A)$ is
$\mathcal{F}_\ell^2$-measurable since, given
$a\in [0,1]$, we have that
\begin{eqnarray*}
\{ (x,y): a \leq (\pi_\ell^x \wedge \pi_\ell^y) (A) \} =  \{ x:
a\leq \pi_\ell^x (A)\}^2.
\end{eqnarray*}
Then both maps $(x,y) \mapsto (\pi^x_\ell - \pi_\ell^x \wedge
\pi_\ell^y)(A)$ and $(x,y) \mapsto (\pi^y_\ell - \pi_\ell^x \wedge
\pi_\ell^y)(A)$ are $\mathcal{F}_\ell^2$-measurable. By (\ref{TV}) also $(x,y) \mapsto d
(\pi_\ell^x,  \pi_\ell^y)$ is measurable in the same sense.

Set $D_\ell = \{(\xi, \xi): \xi \in \Xi_\ell\}$.
Since $\Xi_\ell$ is a standard Borel space, the map $\xi \mapsto
\psi (\xi) = (\xi, \xi)\in D_\ell$ is measurable. Then, for each $B\in
\mathcal{B}(\Xi^2_\ell)$, we have that $\psi^{-1} (B\cap D_\ell) \in
\mathcal{B}(\Xi_\ell)$, which allows us to define $\omega_\ell^{x,y} \in
\mathcal{P}(\Xi^2_\ell)$ by setting
\[
\omega_\ell^{x,y} (B) = (\pi_\ell^x \wedge \pi_\ell^y) \left(
\psi^{-1} (B\cap D_\ell) \right).
\]
The coupling for which (\ref{7}) holds has the form, see
 \cite[Eq. (5.3), page 19]{Lind},
\begin{equation*}
\varrho_\ell^{x,y} = \omega_\ell^{x,y} + (\pi^x_\ell - \pi_\ell^x
\wedge \pi_\ell^y) \otimes (\pi^y_\ell - \pi_\ell^x \wedge
\pi_\ell^y)/ d (\pi_\ell^x,  \pi_\ell^y).
\end{equation*}
Then the $\mathcal{F}_\ell^2$-measurability of
the maps $(x,y) \mapsto \varrho^{x,y}_\ell (A_1 \times A_2)$,
$A_1,A_2 \in \mathcal{B}(\Xi_\ell)$, follows by the arguments given
above. This yields the proof of claim (a) as $\mathcal{B}(\Xi^2_\ell)$ is
a product $\sigma$-field.
\end{proof}

Let $\varpi$ be the family of $\varpi_\ell = \{\varpi_\ell^{x,y}:
(x,y) \in X^2\}$, $\ell \in {\sf L}$, such that each
$\varpi_\ell^{x,y}$ is in $\mathcal{P}(\Xi^2_\ell)$ and, for any $B\in
\mathcal{B}(\Xi^2_\ell)$, the map $(x,y) \mapsto \varpi_\ell^{x,y} (B)$
is $\mathcal{F}_\ell^2$-measurable. Then
$\varpi$ is a one-point specification in the sense of Definition
\ref{1df}, which determines the set $\mathcal{M}(\varpi)$ of $\nu\in
\mathcal{P}(X^2)$ consistent with $\varpi$. Like in (\ref{5}), $\nu \in
\mathcal{M}(\varpi)$ if and only if it satisfies
\begin{eqnarray}
  \label{8}
\nu(B)  & = & \int_{X^2} \int_{X^2} \mathbb{I}_B(x,y)  \varpi_\ell^{y,\tilde{y}}(dx_\ell, d\tilde{x}_\ell)\\[.2cm]
& \times & \prod_{\ell'\neq \ell} \delta_{y_{\ell'}} (d x_{\ell'})\delta_{\tilde{y}_{\ell'}} (d \tilde{x}_{\ell'})\nu(dy, d\tilde{y}),\nonumber
\end{eqnarray}
which holds for all $\ell\in {\sf L}$ and $B\in \mathcal{B}(X^2)$.
\begin{proposition}
  \label{1pn}
Suppose that $\varpi^{x, \tilde{x}}_\ell\in \mathcal{C}(\pi_\ell^x, \pi_\ell^{\tilde{x}})$ for all $\ell\in {\sf L}$ and
$x, \tilde{x}\in X$. Then each $\nu\in \mathcal{M}(\varpi)$ is a coupling of some $\mu_1 , \mu_2 \in \mathcal{M}(\pi)$.
\end{proposition}
\begin{proof}
The equality $\mu_1(A)= \nu(A\times X)$, $A\in \mathcal{B}(X)$, determines a probability measure on $X$.
Thus, for $A\in \mathcal{B}(X)$, by (\ref{8}) we get
\begin{eqnarray*}
\mu_1(A) & = & \int_{X^2} \int_{X^2}\mathbb{I}_A (x) \varpi_\ell^{y,\tilde{y}}(dx_\ell, d\tilde{x}_\ell)
 \prod_{\ell'\neq \ell} \delta_{y_{\ell'}} (d x_{\ell'})\delta_{\tilde{y}_{\ell'}} (d \tilde{x}_{\ell'})\nu(dy, d\tilde{y})\\[.2cm]
& = & \int_{X^2} \int_{X}\mathbb{I}_A (x) \pi_\ell^y (dx_\ell) \prod_{\ell'\neq \ell} \delta_{y_{\ell'}}(dx_{\ell'})\nu(dy, d\tilde{y})\\[.2cm]
& = & \int_{X} \int_{X}\mathbb{I}_A (x) \pi_\ell^y (dx_\ell) \prod_{\ell'\neq \ell} \delta_{y_{\ell'}}(dx_{\ell'})
\int_X \nu(dy, d\tilde{y})\nonumber \\[.2cm]
& = & \int_{X} \int_{X}\mathbb{I}_A (x) \pi_\ell^y (dx_\ell) \prod_{\ell'\neq \ell} \delta_{y_{\ell'}}(dx_{\ell'}) \mu_1(dy).
\end{eqnarray*}
Therefore, $\mu_1$ solves (\ref{5}) and hence $\mu_1\in \mathcal{M}(\pi)$. The same is true for the second marginal measure $\mu_2$.
\end{proof}

\subsection{The results}
Our main concern is under which conditions imposed on the family
$\pi$ the set $\mathcal{M}(\pi)$ contains one element at most. If
each $\pi_\ell^x$ is independent of $x$, the unique element of
$\mathcal{M}(\pi)$ is the product measure $\otimes_{\ell \in {\sf
L}} \pi_\ell$, which readily follows from (\ref{5}). Therefore, one
may try to relate the uniqueness in question to the weak dependence
of $\pi_\ell^x$ on $x$,  formulated in terms of the metric defined
in (\ref{7}). Thus, let us take $x,y\in X$ such that $x=y$ off some
$\ell' \in \partial \ell$, and consider $d(\pi_\ell^x,
\pi_{\ell}^y)$. If this quantity were bounded by a certain
$\kappa_{\ell\ell'}$, uniformly in $x$ and $y$, this bound
(Dobrushin's estimator, cf. \cite[pp. 20, 21]{Beth}) could be used
to formulate the celebrated Dobrushin uniqueness condition in the
form
\begin{equation}
  \label{9}
  \sup_{\ell \in {\sf L}} \sum_{\ell'\in \partial \ell} \kappa_{\ell \ell'} =: \bar{\kappa} < 1.
\end{equation}
However, in a number of applications, especially where $\Xi$ is a noncompact topological space, the mentioned boundedness does not hold.
The way of treating such cases suggested in \cite{DobP} may be outlined as follows. Assume that there exists a matrix $(\kappa_{\ell \ell'})$ with the property as in (\ref{9})
such that, for each $\ell \in {\sf L}$, the following holds
\begin{equation}
  \label{10}
 d(\pi_\ell^x, \pi_{\ell}^y) \leq \sum_{\ell' \in \partial \ell} \kappa_{\ell\ell'} \upsilon(x_{\ell'}, y_{\ell'}),
\end{equation}
for $x$ and $y$ belonging to the set
\begin{equation}
  \label{11}
X_\ell (h,K) := \{ x \in X: h(x_{\ell'}) \leq K \ \ {\rm for} \ {\rm all} \ \ell'  \in \partial \ell\}.
\end{equation}
Here $K>0$ is a parameter and $h:\Xi \to [0,+\infty)$ is a given
measurable function. Clearly, if $h$ is bounded, then $X_\ell
(h,K)=X$ for big enough $K$, and hence (\ref{10}) turns into the
mentioned Dobrushin condition. Thus, in order to cover the case of
interest we have to take $h$ unbounded and  $\pi_\ell^x$-integrable,
with an appropriate control of the dependence of $\pi_\ell^x(h)$ on
$x$. Namely, we shall assume that, for each $\ell\in {\sf L}$ and $x
\in X$, the following holds
\begin{equation}
  \label{12}
 \pi_\ell^x (h) \leq 1 + \sum_{\ell'
 \in \partial \ell} c_{\ell \ell'} h(x_{\ell'}),
\end{equation}
for some  matrix $c=(c_{\ell \ell'})$, which satisfies
\begin{equation}
  \label{13}
(a) \quad  c_{\ell\ell'} \geq 0; \qquad \quad (b) \quad \sup_{\ell
\in {\sf L}} \sum_{\ell' \in \partial \ell} c_{\ell \ell'} =:\bar{c}
< 1/ \varDelta^\chi.
\end{equation}
In the original work \cite{DobP}, the first summand on the
right-hand side of (\ref{12}) is a constant $C>0$, the value of
which  determines the scale of $K$, see (\ref{11}). We thus take it
as above for the sake of convenience.
\begin{definition}
 \label{2df}
Let $h$, $K$, $\kappa$,  and $c$ be as in (\ref{9}) -- (\ref{13}).
Then by $\Pi(h,K,\kappa,c)$ we denote the set of one-site
specifications $\pi$ for which both estimates (\ref{10}), (\ref{11})
and (\ref{12}) hold true for each $\ell \in {\sf L}$.
\end{definition}
Given $\mu\in \mathcal{M}(\pi)$,  the integrability assumed in (\ref{12}) does not yet imply that $h$ is $\mu$-integrable.
For $\pi$ satisfying (\ref{12}),  by
$\mathcal{M}(\pi, h)$ we denote the subset of $\mathcal{M}(\pi)$ consisting of those measures for which the following holds
\begin{equation}
 \label{14}
 \mu(h):=\sup_{\ell \in {\sf L}} \int_{X} h(x_\ell) \mu(d x) < \infty.
\end{equation}
In a similar way, we introduce the set $\mathcal{M}_{\sf D}(\pi, h)$ for a given ${\sf D}\subset {\sf L}$, cf. Definition \ref{1df}.

From now on we fix the graph, the function $h$, and the matrices $c$
and $\kappa$. Thereafter, we set
\begin{equation}
 \label{K}
K_* = \max\left\{\frac{4 \varDelta^{\chi+1}}{\bar{c}(1 - \bar{\kappa})} ; \ \
\frac{2 \varDelta^{\chi+1} (2 \varDelta^{\chi-1} +1-\bar{c} \varDelta^\chi)}{(1-\bar{\kappa})^2( 1- \bar{c} \varDelta^\chi)}\right\}.
\end{equation}
\begin{theorem}
\label{1tm}
For each $K>K_*$ and
$\pi \in  \Pi(h,K,\kappa,c)$, the
set $\mathcal{M}(\pi, h)$ contains at most one element.
\end{theorem}
An important characteristic of the states $\mu\in \mathcal{M}(\pi)$
is the decay of correlations. Fix two distinct vertices $\ell_1,
\ell_2 \in {\sf L}$ and consider bounded functions $f, g:X\to
\mathbb{R}_{+}$, such that $f$ is
$\mathcal{B}(\Xi_{\ell_1})$-measurable and $g$ is
$\mathcal{B}(\Xi_{\ell_2})$-measurable. Set
\[
 {\rm Cov}_\mu (f;g) = \mu(fg) - \mu(f)\mu(g),
\]
and let $\delta$ denote the path distance on the underlying graph.
\begin{theorem}
 \label{2tm}
 Let $\pi$ and $K$ be as in Theorem \ref{1tm}, and $\mathcal{M}(\pi, h)$ be nonempty and hence contain a single state $\mu$.
 Let also $f$ and $g$ be as just described and $\|\cdot \|_\infty$ denote the sup-norm on $X$.
 Then there exist positive $C_K$ and $\alpha_K$, dependent on $K$ only, such that
 \begin{equation}
 \label{dc}
| {\rm Cov}_\mu (f;g)| \leq C_K \|f\|_\infty \|g\|_\infty \exp\left[- \alpha_K \delta(\ell_1 , \ell_2) \right].
 \end{equation}
\end{theorem}

\subsection{Comments and applications}

 \label{222}
 Let us make some comments to the above results. For further comments related to the proof of these results see the end of Section \ref{3SEC}.

\begin{itemize}
 \item According to \cite[Section 8]{Preston}, the elements of $\mathcal{M}(\pi)$ as in Definition \ref{1df}
 are one-site Gibbs states. In \cite[Theorem 1.33, page 23]{Ge} and \cite[Section 8]{Preston}, there are given conditions under which
the elements of $\mathcal{M}(\pi)$ are `usual' Gibbs states, e.g.,
in the sense of \cite[Definition 1.23, page 16]{Ge}. This, in
particular, holds if $\pi$ is a subset of the  set of all local
kernels $\Pi_{\sf D}$ defined for all finite ${\sf D}\subset {\sf
L}$, which determine the states. In this case, Theorem \ref{1tm}
yields the existence and  uniqueness of the usual states, see
\cite{Diana}.
\item The condition in (\ref{14}) is usually satisfied for {\it tempered} measures, i.e., for those elements of $\mathcal{M}(\pi)$
which are supported on {\it tempered} configurations, cf., e.g., \cite{Lebow}.
\item As mentioned above, we do not require that $h$ be {\it compact} in the sense of \cite{DobP}. This our extension gets important if one deals with
single-spin spaces which are not locally compact, e.g., with spaces
of H\"older continuous functions as in \cite{Mon,KP,Pasurek}.
\item In contrast to \cite[Theorem 1]{DobP}, in (\ref{K}) we give an explicit expression for the threshold value $K_*$,
 which  depends only on the parameters of the underlying graph and on the norms $\bar{c}$ and $\bar{\kappa}$.
\item The novelty of Theorem \ref{2tm} consists in the following. The decay of correlations under the uniqueness condition was
proven only for compact single-spin spaces, see \cite{Ku}, where the
classical Dobrushin criterion can be applied. For `unbounded spins',
the corresponding results are usually obtained by cluster
expansions, see, e.g., \cite{PS}, where the  correlations  are shown
to decay due to `weak enough' interactions' and no information on
the number of states is available.
\item The parameters $C_K$ and $\alpha_K$ in (\ref{dc}) are also given explicitly, see (\ref{constan}) below.

\end{itemize}

Now we turn to briefly outlining possible applications of Theorems
\ref{1tm} and \ref{2tm}. A more detailed discussion of this issue
can be found in \cite{Diana}, see also the related parts of
\cite{Pasurek}. Further results in these directions will be
published in forthcoming articles.

By means of Theorems \ref{1tm} and \ref{2tm} the uniqueness of
equilibrium states and the decay of correlations can be established
in the following models:
\begin{itemize}
  \item Systems of classical $N$-dimensional anharmonic oscillators described by the
  energy functional
\[
H(x) = \sum_{\ell \in {\sf L}} V(\xi_\ell) + \sum_{(\ell, \ell')\in
{\sf E}} W_{\ell \ell'} (\xi_\ell, \xi_{\ell'}), \qquad \xi_\ell \in
\mathbb{R}^N, \ N \in \mathbb{N}
\]
\item Systems of quantum $N$-dimensional anharmonic oscillators described by the
Hamiltonian
\[
H = \sum_{\ell \in {\sf L}} H_\ell + \sum_{(\ell, \ell')\in {\sf E}}
W_{\ell \ell'} (q_\ell, q_{\ell'}),
\]
where $q_\ell= (q^{(1)}_\ell, \dots , q^{(N)}_\ell)$ is the position
operator and $H_\ell$ is the one-particle Hamiltonian defined on the
corresponding physical Hilbert space. States of such models are
constructed in a path integral approach as probability measures on
the products of continuous periodic functions, which are not locally
compact, see \cite{Mon,KP,Pasurek}.
\item Systems of interacting particles in the continuum (e.g.
$\mathbb{R}^d$), including the Lebowitz-Mazel-Presutti model
\cite{LMP}, and systems of `particles' lying on the cone of discrete
measures introduced in \cite{Hagedorn}. Note that to continuum
systems the original version \cite{DobP} of the Dobrushin-Pechersky
criterion was used in \cite{BP,Pechersky}.

\end{itemize}

\section{The Proof of Theorems \ref{1tm} and \ref{2tm}}
\label{3SEC}

\subsection{The ingredients of the proof}

First we introduce the notion of {\it locality}. By writing ${\sf D}\Subset {\sf L}$ we mean that $\sf D$ is
a nonempty finite subset  of $\sf L$.
For such $\sf D$, elements of $\mathcal{B}(\Xi^{\sf D})\subset \mathcal{B}(X)$ are called local sets.
A function $f:X \to \mathbb{R}$ is called local if it is $\mathcal{B}(\Xi^{\sf D})$-measurable for some ${\sf D}\Subset {\sf L}$.
Likewise, $B\in \mathcal{B}(X^2)$ is local if $B\in\mathcal{B}((\Xi\times \Xi)^{\sf D})$ for such ${\sf D}$. Locality of functions
$f:X^2 \to \mathbb{R}$ is defined in the same way.
\begin{lemma}
 \label{1lm}
Given a one-site specification $\pi$  and $\mu_1, \mu_2 \in \mathcal{M}(\pi)$, suppose there exists $\nu_*\in \mathcal{C}(\mu_1, \mu_2)$
such that
\begin{equation}
 \label{15}
 \int_{X^2} \upsilon (x_\ell , y_\ell) \nu_* (d x, dy) = 0,
\end{equation}
holding for all $\ell \in {\sf L}$. Then $\mu_1 = \mu_2$.
\end{lemma}
\begin{proof}
Local sets $A\subset X$ are measure defining, that is, $\mu_1 , \mu_2 \in \mathcal{P}(X)$ coincide if they coincide on local sets.
For $A\in \mathcal{B}(\Xi^{\sf D})$ and the indicator $\mathbb{I}_A$, we have
\[
|\mathbb{I}_A (x) - \mathbb{I}_A (y)| \leq \sum_{\ell \in {\sf D}} \upsilon (x_\ell, y_\ell),
\]
and then
\[
|\mu_1(A) - \mu_2 (A) | \leq \sum_{\ell \in {\sf D}} \int_{X^2} \upsilon (x_\ell, y_\ell) \nu_* (d x, dy) = 0,
\]
which yields the proof.
\end{proof}
The proof of Theorem \ref{1tm} will be done by showing that, for
each $\mu_1 , \mu_2 \in \mathcal{M}(\pi,h)$, the set
$\mathcal{C}(\mu_1 , \mu_2)$ contains a certain $\nu_*$ such that
(\ref{15}) holds. This coupling $\nu_*$ will be obtained by taking
the limit in the topology of local setwise convergence, cf.
\cite{Ge}, which we introduce as follows.
\begin{definition}
 \label{3df}
A net $\{\nu_\alpha\}_{\alpha \in I} \subset \mathcal{P}(X^2)$ is said to be convergent to a $\nu_*\in \mathcal{P}(X^2)$ in the topology
of local setwise convergence ($\mathfrak{L}$-topology, for short), if
$\nu_\alpha (B) \to \nu_* (B)$ for all local $B\in \mathcal{B}(X^2)$. Or, equivalently, $\nu_\alpha (f) \to \nu_*(f)$ for all bounded local functions.
The same definition applies also to nets $\{\mu_\alpha\}_{\alpha \in I} \subset \mathcal{P}(X)$.
\end{definition}
Note that the $\mathfrak{L}$-topology is Hausdorff, but not metrizable if $\Xi$ is not a compact topological space.
\begin{lemma}
 \label{2lm}
Given $\mu_1 , \mu_2 \in \mathcal{P}(X)$, let $\{\nu_\alpha \}_{\alpha \in I} \subset \mathcal{C}(\mu_1 , \mu_2)$ be convergent
to a certain $\nu\in \mathcal{P}(X^2)$ in the $\mathfrak{L}$-topology. Then $\nu\in \mathcal{C}(\mu_1 , \mu_2)$.
\end{lemma}
The proof of this lemma is rather obvious.  The coupling in question $\nu_*$ will be constructed within a step-by-step procedure
based on the mapping
\begin{equation}
\label{16}
(R_\ell \nu)(f) = \int_{X^2} \left( \int_{\Xi^2} f( \xi\times x_{\{\ell\}^c}, \eta\times y_{\{\ell\}^c })
\varrho^{x,y}_\ell (d \xi, d\eta)\right) \nu( d x, d y),
\end{equation}
where $\ell \in {\sf L}$,  $\varrho_\ell^{x,y}$ is as in (\ref{7}), and $f:X^2 \to \mathbb{R}$ is a function such that
both $\nu(f)$ and the integral on the right-hand side of (\ref{16}) exist.
\begin{lemma}
 \label{Rlm}
For each $\ell  \in {\sf L}$, the mapping (\ref{16}) has the following properties: (a)
if $\nu\in \mathcal{C}(\mu_1, \mu_2)$ for some $\mu_1, \mu_2\in \mathcal{M}(\pi)$, then
also $R_\ell \nu \in \mathcal{C}(\mu_1, \mu_2)$; (b) if $f$ is $\mathcal{F}_\ell (X^2)$-measurable and $\nu$-integrable,
then $(R_\ell \nu)(f) = \nu(f)$.
\end{lemma}
\begin{proof}
Claim (a) is true since $\varrho_\ell^{x,y} \in \mathcal{C}(\pi_\ell^x, \pi_\ell^y)$ for all $x, y \in X$. Claim
(b) follows by the fact that the considered $f$ in (\ref{16}) is independent of $\xi$ and $\eta$, and that $\varrho_\ell^{x,y}$
is a probability measure.
\end{proof}
Given $\ell \in {\sf L}$, we set
\begin{equation*}
Y_\ell = \{(x^1, x^2)\in X^2: \upsilon (x_\ell^1, x_\ell^2) \leq \sum_{\ell'\in \partial \ell} \upsilon (x_{\ell'}^1, x_{\ell'}^2) \}.
\end{equation*}
\begin{lemma}
 \label{R1lm}
For each $\nu\in \mathcal{P}(X^2)$ and  $\ell \in {\sf L}$, it follows that $(R_\ell \nu)(Y_\ell) =1$.
\end{lemma}
\begin{proof}
If $(x^1, x^2)$ is in $Y_\ell^c$, then  $\upsilon (x_\ell^1,
x_\ell^2) =1$ and $\upsilon (x_{\ell'}^1, x_{\ell'}^2)=0$ for all
$\ell'\in \partial \ell$, which follows by the fact that $\upsilon$
takes values in $\{0,1\}$. This means that $x_\ell^1 \neq x_\ell^2$
and $x_{\ell'}^1 = x_{\ell'}^2$ for all $\ell'\in \partial \ell$.
For such $(x^1, x^2)$, the definition of $\pi$ implies that
$\pi_\ell^{x^1} = \pi_\ell^{x^2}$, and hence
\[
\int_{\Xi^2} \upsilon (\xi, \eta) \varrho^{x^1,x^2}_\ell (d\xi,
d\eta) = d(\pi^{x_1}_\ell, \pi^{x_2}_\ell) = 0,
\]
which by (\ref{16}) yields $(R_\ell \nu)(Y_\ell^c) = 0$.
\end{proof}
The proof of Theorem \ref{1tm} will be done by showing that, for each $\mu_1, \mu_2 \in \mathcal{M}(\pi, h)$, there exists
$\nu_* \in \mathcal{C}(\mu_1, \mu_2)$, for wich (\ref{15}) holds.
To this end we construct a sequence $\{\hat{\nu}_n\}_{n\in \mathbb{N}_0} \subset
\mathcal{C}(\mu_1, \mu_2)$ such that
\begin{equation}
\label{18}
\gamma(\hat{\nu}_n):= \sup_{\ell \in {\sf L}} \int_{X^2} \upsilon (x^1_\ell , x_\ell^2) \hat{\nu}_n (d x^1 , d x^2) \to 0, \qquad n \to +\infty.
\end{equation}
This sequence will be obtained by a procedure based on the mapping (\ref{16}) and the estimates which we derive in the next subsection.
The proof of Theorem \ref{2tm} will be obtained as a byproduct.

\subsection{The main estimates}

In the sequel, we use the following functions indexed by $\ell \in {\sf L}$
\begin{equation}
 \label{19}
I_\ell (x^1, x^2) =  \upsilon (x^1_\ell , x_\ell^2), \qquad H^i_\ell (x^1, x^2) = h(x_\ell^i), \ \  i = 1,2.
\end{equation}
By claim (b) of Lemma \ref{Rlm}, we have that
\begin{equation*}
(R_\ell \nu)(I_{\ell_1}) = \nu (I_{\ell_1}), \quad (R_\ell \nu)(I_{\ell_1}H^i_{\ell_2}) =
\nu(I_{\ell_1}H^i_{\ell_2}) \quad {\rm for} \ \ell \neq \ell_1, \ \ell \neq \ell_2,
\end{equation*}
whenever $H^i_\ell$ is $\nu$-integrable.
We recall that $\varrho_\ell^{x,y}$ in (\ref{16}) is a coupling of $\pi_\ell^x$ and $\pi_\ell^y$, for which (\ref{10})
and (\ref{12}) hold true.
\begin{lemma}
 \label{R2lm}
Let $\nu \in \mathcal{P}(X^2)$ be such that the integrals on both sides of (\ref{16}) exist for $f=H^i_\ell$, $\ell \in {\sf L}$ and $i=1,2$.
Then the following estimates hold
\begin{gather}
 \label{21}
 (R_\ell \nu)(I_{\ell}) \leq \sum_{\ell' \in\partial \ell} \kappa_{\ell \ell'} \nu(I_{\ell'}) +
K^{-1} \sum_{i=1,2} \sum_{\ell_1 , \ell_2 \in \partial \ell} \nu(I_{\ell_2}H^i_{\ell_1}), \\[.2cm]
\label{22}
(R_\ell \nu)(I_{\ell_1}H^i_{\ell}) \leq \nu (I_{\ell_1}) + \sum_{\ell_2 \in \partial \ell}c_{\ell\ell_2}
\nu(I_{\ell_1}H^i_{\ell_2}),\\[.2cm]
\label{23}
(R_\ell \nu)(I_{\ell}H^i_{\ell_1}) \leq \sum_{\ell_2 \in \partial \ell} \nu(I_{\ell_2}H^i_{\ell_1}), \qquad \ell_1 \neq \ell, \\[.3cm]
\label{24}
(R_\ell \nu)(I_{\ell}H^i_{\ell}) \leq  \sum_{\ell_1 \in \partial \ell} \nu(I_{\ell_1}) +
\sum_{\ell_1 , \ell_2 \in \partial \ell}c_{\ell \ell_2} \nu(I_{\ell_1}H^i_{\ell_2}).
\end{gather}
\end{lemma}
\begin{proof}
The proof of (\ref{23}) readily follows by Lemma \ref{R1lm}. Let us prove
(\ref{21}). By (\ref{7}) and (\ref{16}), we have
\begin{eqnarray*}
(R_\ell \nu)(I_{\ell}) & = & \int_{X^2} d(\pi_\ell^{x^1}, \pi_\ell^{x^2}) \nu(dx^1 , dx^2)\\[.2cm]
& = & \int_{X^2} \mathbf{1}_\ell (x^1)\mathbf{1}_\ell (x^2)  d(\pi_\ell^{x^1}, \pi_\ell^{x^2}) \nu(dx^1 , dx^2)\nonumber \\[.2cm]
& + & \int_{X^2} \left[ 1 - \mathbf{1}_\ell (x^1)\mathbf{1}_\ell (x^2) \right]  d(\pi_\ell^{x^1}, \pi_\ell^{x^2}) \nu(dx^1 , dx^2) ,\nonumber
\end{eqnarray*}
where $\mathbf{1}_\ell$ is the indicator of the set defined in (\ref{11}). By (\ref{10}), we have
\[
\int_{X^2} \mathbf{1}_\ell (x^1)\mathbf{1}_\ell (x^2)  d(\pi_\ell^{x^1}, \pi_\ell^{x^2}) \nu(dx^1 , dx^2) \leq
\sum_{\ell' \in\partial \ell} \kappa_{\ell \ell'} \nu(I_{\ell'}),
\]
which yields the first term of the right-hand side of (\ref{21}). By (\ref{11}), we have
\[
\left[ 1 - \mathbf{1}_\ell (x^1)\mathbf{1}_\ell (x^2) \right]  \leq \sum_{i=1,2} \sum_{\ell_1 \in \partial \ell}
\left[1 - \mathbb{I}_{h\leq K} (x_{\ell_1}^i)\right],
\]
where $\mathbb{I}_{h\leq K}$ is the indicator of $\{\xi \in \Xi: h(\xi) \leq K\}$. Then the second term
of the right-hand side of (\ref{21}) cannot exceed the following
\begin{eqnarray*}
& & \sum_{i=1,2} \sum_{\ell_1 \in \partial \ell}  \int_{X^2}
\left[1 - \mathbb{I}_{h\leq K} (x_{\ell_1}^i)\right] d(\pi_\ell^{x^1}, \pi_\ell^{x^2}) \nu(dx^1 , dx^2)\\[.2cm]
& &  \qquad \leq K^{-1} \sum_{i=1,2} \sum_{\ell_1 \in \partial \ell}  \int_{X^2} h(x^i_{\ell_1})
 d(\pi_\ell^{x^1}, \pi_\ell^{x^2}) \nu(dx^1 , dx^2)\\[.2cm]
& & \qquad \leq  K^{-1} \sum_{i=1,2} \sum_{\ell_1 , \ell_2\in \partial \ell} \nu(I_{\ell_2} H^i_{\ell_1}).
\end{eqnarray*}
The latter line has been obtained by (\ref{23}).

Let us prove now (\ref{22}). By (\ref{16}) and the fact that
$\varrho_\ell^{x,y} \in \mathcal{C}(\pi_\ell^x, \pi_\ell^y)$, we have
\begin{eqnarray*}
 (R_\ell \nu)(I_{\ell_1}H^i_{\ell}) & = & \int_{X^2} \left( \int_{\Xi} h (\xi) \pi^{x^i} (d\xi) \right)
\upsilon(x^{1}_{\ell_1} , x^{2}_{\ell_1}) \nu(dx^1, dx^2)\\[.2cm]
& \leq & {\rm RHS}(\ref{22}),
\end{eqnarray*}
where we have used (\ref{12}). To prove (\ref{24}) we employ Lemma \ref{R1lm}, by which we get
\[
{\rm LHS}(\ref{24}) \leq \sum_{\ell_1 \in \partial \ell} (R_\ell \nu) (I_{\ell_1} H^i_\ell) \leq {\rm RHS}(\ref{24}) ,
\]
where the latter estimate follows by (\ref{22}).
\end{proof}
From the lemma just proven it follows  that along with the parameter $\gamma(\nu)$ defined in (\ref{18}) one has to control
also the following
\begin{equation}
 \label{26}
\lambda (\nu) = \max_{i=1,2} \sup_{\ell, \ell' \in {\sf L}} \nu(I_\ell H^i_{\ell'}),
\end{equation}
where $\nu\in \mathcal{C}(\mu_1 , \mu_2)$, $\mu_1, \mu_2 \in \mathcal{M}_h(\pi)$, and $\pi \in \Pi(h,K,\kappa, c)$, see Definition \ref{2df}.

\subsection{The proof of Theorem \ref{1tm}}

The proof is based on constructing a sequence with the property (\ref{18}). Given $\mu_1, \mu_2 \in \mathcal{M}(\pi,h)$ with
$\pi \in \Pi(h,K,\kappa, c)$, we take an arbitrary $\nu_0\in \mathcal{C}(\mu_1, \mu_2 )$ and construct $\nu \in \mathcal{C}(\mu_1, \mu_2 )$ by applying
the mapping defined in (\ref{16}) to the initial $\nu_0$ with $\ell$ running over the set ${\sf L}$.
Each time we use the estimates derived in Lemma \ref{R2lm}. Then the first two elements of the sequence in question are set $\hat{\nu}_0 = \nu_0$ and
$\hat{\nu}_1 = \nu$. Afterwards, we produce $\hat{\nu}_2$ from $\hat{\nu}_1$, etc.

Recall that the underlying graph is supposed to have the property defined in (\ref{1}) and
$\chi\leq \varDelta$
is its chromatic number. Set
\begin{equation}
 \label{26a}
 A = \frac{2 \varDelta^{\chi +1}}{1 - \bar{\kappa}}.
\end{equation}
Then, for $K>K_*$, see (\ref{K}), the following holds
\begin{equation}
 \label{26b}
K^{-1} < \frac{\bar{c}(1 - \bar{\kappa})}{4 \varDelta^{\chi+1}}, \qquad A K^{-1} < \bar{c}/2.
\end{equation}
\begin{lemma}
 \label{R3lm}
For $K> K_*$, take $\pi \in \Pi(h,K,\kappa, c)$  and $\mu_1 , \mu_2
\in \mathcal{M}(\pi, h)$. Then   for each $\nu_0 \in
\mathcal{C}(\mu_1 , \mu_2)$ there exists $\nu \in \mathcal{C}(\mu_1
, \mu_2)$ for which the following estimates hold
\begin{gather}
 \label{27}
 \gamma(\nu) \leq \left[\bar{\kappa} + AK^{-1} \right] \gamma(\nu_0) + 2 A K^{-1} \lambda (\nu_0),\\[.3cm]
 \label{28}
 \lambda (\nu) \leq \Delta^{\chi-1} \gamma(\nu_0) + \bar{c} \varDelta^\chi  \lambda(\nu_0).
\end{gather}
\end{lemma}
The proof of the lemma will be given in the subsequent parts of the paper.
\vskip.1cm \noindent
{\it Proof Theorem \ref{1tm}:}
As already mentioned, we let $\hat{\nu}_1\in \mathcal{C}(\mu_1, \mu_2)$ and
$\hat{\nu}_0\in \mathcal{C}(\mu_1, \mu_2)$ be the measures on
the left-hand sides and right-hand sides of (\ref{27}) and (\ref{28}), respectively. Then we
 apply to $\hat{\nu}_1$ the same reconstruction procedure and obtain $\hat{\nu}_2 \in \mathcal{C}(\mu_1, \mu_2)$, for which both estimates
(\ref{27}), (\ref{28}) hold with $\hat{\nu}_1$ on the right-hand sides. We repeat this due times and obtain
$\hat{\nu}_n \in \mathcal{C}(\mu_1, \mu_2)$
such that
\begin{equation}
 \label{28a}
\left(\begin{array}{ll} \gamma(\hat{\nu}_n)\\[.3cm] \lambda (\hat{\nu}_n) \end{array}
                         \right)  \leq\left[ M(K) \right]^n \left(\begin{array}{ll} \gamma(\nu_0)\\[.3cm] \lambda (\nu_0) \end{array}
                         \right),
\end{equation}
where $M(K)$ is the matrix defined by the right-hand sides of (\ref{27}) and (\ref{28}).
Its spectral radius is
\begin{equation}
 \label{srM}
r_K = \frac{1}{2}\left[\bar{\kappa} + A K^{-1} + \bar{c}\varDelta^\chi + \sqrt{(\bar{\kappa} + A K^{-1} - \bar{c}\varDelta^\chi)^2 +
8 \varDelta^{\chi} AK^{-1}} \right].
\end{equation}
For $K>K_*$, see (\ref{K}),
we have $r_K < 1$, which by (\ref{28a}) yields (\ref{18}) and thereby completes the proof.

\subsection{The proof of Theorem \ref{2tm}}

The proof of this theorem is based on the version of the estimates in Lemma \ref{R3lm}
obtained in a subset ${\sf D}\subset {\sf L}$. For such $\sf D$, we define
\begin{equation*}
\partial {\sf D} = \{\ell' \in {\sf D}^c: \partial \ell' \cap {\sf D}\neq \emptyset\},
\end{equation*}
which is the external boundary of ${\sf D}$. For $\nu \in \mathcal{P}(X^2)$
such that all $H^{i}_\ell$, $i=1,2$, $\ell \in {\sf D}\cup \partial {\sf D}$ are
$\nu$-integrable,  see (\ref{19}), we set, cf. (\ref{18}) and (\ref{26}),
\begin{equation}
 \label{dc2}
\gamma_{\sf D} (\nu) = \sup_{\ell \in {\sf D}} \nu(I_\ell), \qquad \lambda_{\sf D}(\nu) = \max_{i=1,2}\sup_{\ell_1, \ell_2 \in {\sf D}}
\nu(I_{\ell_1} H^{i}_{\ell_2}).
\end{equation}
Next, for $\ell_1$ as in (\ref{dc}) and $N= \delta (\ell_1 , \ell_2)$, we set
\begin{equation*}
{\sf D}_0 = \{\ell_1\}, \quad {\sf D}_{k} = {\sf D}_{k-1} \cup \partial {\sf D}_{k-1}, \quad k=1, \dots , N-1.
\end{equation*}
Let $\mu^x(\cdot)$ denote the conditional measure $\mu
(\cdot|\mathcal{F}_{{\sf D}_{N-1}})(x)$. For brevity, we say that
$\nu^x \in \mathcal{P}(X^2)$ is $\mathcal{F}_{{\sf
D}_{N-1}}$-measurable if the maps $x\mapsto \nu^x(B)$ are
$\mathcal{F}_{{\sf D}_{N-1}}$-measurable for all $B\in
\mathcal{B}(X^2)$. Clearly, $\nu_0^x = \mu^x \otimes \mu$ possesses
this property. The version of Lemma \ref{R3lm} which we need is the
following statement.
\begin{lemma}
 \label{dclm}
Let $\pi$, $K$, and $\mu$  be as in Theorem \ref{2tm} and $\nu_0^x =
\mu^x \otimes \mu$. Then there exist $\nu_1^x, \dots , \nu_{N-1}^x
\in \mathcal{C}(\mu^x, \mu)$, all $\mathcal{F}_{{\sf
D}_{N-1}}$-measurable, such that for the parameters defined in
(\ref{dc2}) the following estimates hold
\begin{equation}
 \label{dc3}
\left(\begin{array}{ll} \gamma_{{\sf D}_{N-s-1}}(\nu^x_{s})\\[.3cm] \lambda_{{\sf D}_{N-s-1}} (\nu^x_s) \end{array}
                         \right)  \leq M(K)  \left(\begin{array}{ll} \gamma_{{\sf D}_{N-s}}(\nu^x_{s-1})\\[.3cm]
                         \lambda_{{\sf D}_{N-s}} (\nu^x_{s-1}) \end{array}
                         \right),
\end{equation}\
for all $s=1, \dots , N-1$ and $\mu$-almost all $x\in X$.
 \end{lemma}
\vskip.1cm \noindent
{\it Proof of Theorem \ref{2tm}:}  Since  $g$ is $\mathcal{F}_{{\sf D}_{N-1}}$-measurable,
we have
\begin{equation*}
\int_{X} f(x) g(x) \mu(dx) = \int_{X} g(x)\left(\int_X f(y) \mu^x(d y) \right)\mu(dx),
\end{equation*}
which yields
\begin{equation}
  \label{dc6}
 {\rm Cov}_\mu (f;g) = \int_X g(x) \Phi(x)\mu(dx),
\end{equation}
where
\begin{eqnarray}
  \label{dc7}
 \Phi (x)  =  \int_{X^2} \left( f(y)-f(z)\right)\mu^x (dy) \mu(dz).
\end{eqnarray}
For each $\nu^x_s$, $s=0, \dots , N-1$, as in Lemma \ref{dclm}, we then have
\begin{equation}
  \label{dc8}
 \Phi(x) =  \int_{X^2} \left( f(y)-f(z)\right) \nu^x_s(d y, dz),
\end{equation}
and hence
\begin{equation}
  \label{dc9}
  \left\vert \Phi(x) \right\vert \leq 2 \|f\|_{\infty} \nu^x_{N-1}(I_{\ell_1}) = 2 \|f\|_{\infty} \gamma_{{\sf D}_0} (\nu^x_{N-1}).
\end{equation}
Note that the function defined in (\ref{dc7}), (\ref{dc8}) is related to the quantity which characterizes mixing in
state $\mu$, cf. \cite[Proposition 2.5]{Ku}.

Let $v_s$ and $v_{s-1}$ denote the column vector on the left-hand and right-hand sides of (\ref{dc3}), respectively. Set
\begin{equation*}
 \xi = \frac{\varDelta^{\chi-1}}{r_K - \bar{c}\varDelta^\chi} = \frac{r_K - \bar{\kappa} - AK^{-1}}{2 AK^{-1}} >0,
\end{equation*}
and let $T$ be the $2\times 2$ diagonal matrix with $T_{11} = \xi$ and $T_{22} =1$.  Then the matrix
\begin{equation}
 \label{dc11}
 \widetilde{M}(K) := T M(K) T^{-1},
\end{equation}
cf. \cite[Corollary 2.9.4, page 102]{BL},
is positive and such that both its rows sum up  to $r_K$. Set $\tilde{v}_s = T v_s$ and let $\tilde{v}_s^{i}$, $i=1,2$,
be the entries of $\tilde{v}_s$. By (\ref{dc3}) we then get
\[
 \|\tilde{v}_s\|:= \max\{\tilde{v}_s^{1}; \tilde{v}_s^{2}\} \leq \|\widetilde{M}(K)\| \|\tilde{v}_{s-1}\| = r_K
 \max\{\tilde{v}_{s-1}^{1}; \tilde{v}_{s-1}^{2}\},
\]
which yields
\begin{equation}
 \label{dc12}
 \gamma_{{\sf D}_0} (\nu^x_{N-1}) \leq r_K^{N-1} \max\{ \gamma_{{\sf D}_{N-1}}(\nu^x_0); \xi^{-1} \lambda_{{\sf D}_{N-1}}(\nu^x_0)\}.
\end{equation}
Applying this estimate in (\ref{dc9}) and then in (\ref{dc6}) we arrive at (\ref{dc}) with,
cf. (\ref{srM}) and (\ref{14}),
\begin{equation}
  \label{constan}
\alpha_K = - \log r_K, \qquad  C_K = 2 r_K^{-1} \max\{1; \xi^{-1} \mu(h)\}.
\end{equation}
Let us make now  further comments on the above results and their proof.
\begin{itemize}
  \item
The mapping  in (\ref{16}), which is the main reconstruction tool, see Section \ref{4SEC} below, was first introduced in another seminal paper by R. L. Dobrushin \cite{Dob}.
In a rather general context, it was used in \cite{dl}. The main feature of this mapping, which was not pointed out in \cite{DobP}, is the
measurability of the coupling $\varrho_\ell^{x,y}$ in $(x,y)\in X^2$. A similar property of the couplings in Lemma \ref{dclm} was crucial
for the proof of Theorem \ref{2tm}.
\item We avoid using `compactness' of $h$, and hence the related topological properties of the single-spin space $\Xi$, by employing the $\mathfrak{L}$-topology, see Definition \ref{3df}.
\item In contrast to the estimates obtained in \cite[Lemma 5]{DobP}, our estimate in (\ref{28}) is independent of $K$. The only constant in (\ref{27}) is given explicitly in (\ref{26a}). This allowed us to calculate explicitly the spectral radius (\ref{srM}), which was then used to obtain the decay parameter $\alpha_K$, see (\ref{constan}).
\item The proof of Lemma \ref{dclm} was performed in the spirit of the proof of Proposition 2.5 of \cite{Ku}. Our $\Phi(x)$ in (\ref{dc7}), (\ref{dc8}) can be used to prove a kind of mixing in state $\mu$. However, here we cannot estimate this function uniformly in $x$, and hence   employ its measurable estimate (\ref{dc9}) which is then integrated in (\ref{dc6}).
\item The transformation used in (\ref{dc11}) allowed us to find explicitly the operator norm of $M(K)$ equal to its spectral radius $r_K$.  This then was used to find in (\ref{dc12}) the exact rate of the decay of correlations in $\mu$.
\end{itemize}

\section{Proof of Lemmas \ref{R3lm} and \ref{dclm}}
\label{4SEC}

For the partition (\ref{3}) of the set of vertices ${\sf L}$, which has the property (\ref{2}), we set
\begin{equation}
 \label{29}
{\sf U}_j = \bigcup_{i=0}^j {\sf V}_i, \qquad {\sf W}_j = {\sf L}\setminus {\sf U}_j, \quad  j=0, \dots , \chi-1.
\end{equation}
The measure $\nu$ in (\ref{27}), (\ref{28}) will be obtained in the course of consecutive reconstructions with $\ell\in {\sf V}_j$. The first step is
\vskip.1cm
\subsection{Reconstruction over ${\sf V}_0$}
\label{sec:1}
Let $\{\ell_1, \ell_2, \dots , \}$ be any numbering of the elements of ${\sf V}_0$. Set
\begin{equation}
 \label{30}
{\sf V}_0^{(n)} = \{\ell_1, \dots , \ell_n\}, \qquad \nu_0^{(n)} = R_{\ell_n} R_{\ell_{n-1}}\cdots R_{\ell_1} \nu_0, \quad n \in \mathbb{N}.
\end{equation}
Our first task is to estimate $\nu_0^{(n)} (I_\ell)$. By claim (b) of Lemma \ref{Rlm} we have that
\begin{equation}
 \label{31}
 \nu_0^{(n)} (I_\ell) = \nu_0 (I_\ell), \qquad {\rm for} \  \ \ell \notin  {\sf V}_0^{(n)}.
\end{equation}
For $k\leq n$, by (\ref{2}) and claim (b) of Lemma \ref{Rlm}, and then by (\ref{21}) and (\ref{31}), we have
\begin{eqnarray}
 \label{32}
\nu_0^{(n)} (I_{\ell_k}) = \nu_0^{(k)} (I_{\ell_k}) & \leq & \sum_{\ell \in \partial \ell_k} \kappa_{\ell_k\ell} \nu_0(I_{\ell}) +
K^{-1}  \sum_{i=1,2} \sum_{\ell, \ell' \in \partial \ell_k} \nu_0 (I_\ell H^i_{\ell'})\nonumber \\[.2cm]
& \leq & \bar{\kappa} \gamma(\nu_0) + 2 \varDelta^2 K^{-1} \lambda (\nu_0),
\end{eqnarray}
see also (\ref{9}), (\ref{18}), and (\ref{26}).

Next we turn to estimating $\nu_0^{(n)} (I_\ell H^i_{\ell'})$. As in (\ref{31}) we have
\begin{equation*}
 \nu_0^{(n)} (I_\ell H^i_{\ell'}) = \nu_0 (I_\ell H^i_{\ell'}) \qquad {\rm for} \  \ \ell, \ell' \notin  {\sf V}_0^{(n)}.
\end{equation*}
For $k<m \leq n$, by claim (b) of Lemma \ref{Rlm}, and then by (\ref{22}), (\ref{21}),
(\ref{32}), and (\ref{23}), we have
\begin{eqnarray*}
 & & \nu_0^{(n)} (I_{\ell_k} H^i_{\ell_{m}}) = \nu_0^{(m)} (I_{\ell_k} H^i_{\ell_{m}})
 \leq  \nu_0^{(k)}(I_{\ell_k}) + \sum_{\ell \in \partial \ell_m} c_{\ell_m \ell} \nu^{(k)}_0(I_{\ell_k}H^i_{\ell}) \nonumber \\[.2cm]
 & & \qquad \leq \bar{\kappa} \gamma(\nu_0) + 2 \varDelta^2 K^{-1} \lambda (\nu_0)  + \sum_{\ell \in \partial \ell_m} c_{\ell_m \ell}
 \sum_{\ell' \in \partial \ell_k} \nu_0 (I_{\ell'} H^i_{\ell}) \nonumber \\[.2cm]
 & & \qquad  \leq  \bar{\kappa} \gamma(\nu_0) + \left[\varDelta \bar{c} + 2 \varDelta^2  K^{-1} \right] \lambda (\nu_0).
 \end{eqnarray*}
For $k\leq n$, by (\ref{24}) we have
\begin{eqnarray*}
 \nu_0^{(n)} (I_{\ell_k} H^i_{\ell_{k}}) = \nu_0^{(k)} (I_{\ell_k} H^i_{\ell_{k}}) & \leq &
 \sum_{\ell \in \partial \ell_k}  \nu_0(I_\ell) +
 \sum_{\ell, \ell' \in \partial \ell_k} c_{\ell_k \ell'} \nu_0( I_\ell H^i_{\ell'}) \nonumber \\[.2cm]
& \leq & \varDelta \gamma(\nu_0) + \varDelta \bar{c} \lambda (\nu_0).
 \end{eqnarray*}
Next, for $m < k \leq n$, by (\ref{23}) and (\ref{22}) we have
\begin{eqnarray*}
 \nu_0^{(n)} (I_{\ell_k} H^i_{\ell_{m}})& = & \nu_0^{(k)} (I_{\ell_k} H^i_{\ell_{m}})
 \leq \sum_{\ell \in \partial \ell_k} \nu_0^{(m)} (I_\ell H^i_{\ell_m}) \nonumber \\[.2cm]
 & \leq & \sum_{\ell \in  \partial \ell_k} \left( \nu_0 (I_\ell) + \sum_{\ell' \in \partial \ell_m}
 c_{\ell_m \ell'} \nu_0 (I_{\ell} H^i_{\ell'}) \right) \nonumber \\[.2cm]
 & \leq & \varDelta \gamma(\nu_0) + \varDelta \bar{c} \lambda (\nu_0).
\end{eqnarray*}
Now we consider the case where $k\leq n$ and $\ell \notin  {\sf V}_0^{(n)}$. Then by (\ref{23}) we have
\begin{equation*}
\nu_0^{(n)} (I_{\ell_k} H^i_{\ell}) = \nu_0^{(k)} (I_{\ell_k} H^i_{\ell})
 \leq  \sum_{\ell' \in \partial \ell_k} \nu_0 (I_{\ell'} H^i_\ell) \leq \varDelta \lambda (\nu_0).
\end{equation*}
For $k\leq n$ and $\ell \notin  {\sf V}_0^{(n)}$, we also have by (\ref{22}) that
\begin{eqnarray}
  \label{38}
   \nu_0^{(n)} (I_{\ell} H^i_{\ell_{k}}) = \nu_0^{(k)} (I_{\ell} H^i_{\ell_{k}}) & \leq & \nu_0 (I_\ell) + \sum_{\ell' \in \partial \ell_k}
c_{\ell_k \ell'} \nu_0 (I_\ell H^i_{\ell'}) \nonumber \\[.2cm]
& \leq  & \gamma(\nu_0) + \bar{c} \lambda (\nu_0).
\end{eqnarray}
Now let us consider the sequence $\{\nu_0^{(n)}\}_{n\in \mathbb{N}_0}$ defined in (\ref{30}).
By claim (b) of Lemma \ref{Rlm} it stabilizes on local sets $B\in \mathcal{B}(X^2)$, and hence is convergent in the $\mathfrak{L}$-topology.
Let $\nu_1$ be its limit. By Lemma \ref{2lm} we have that $\nu_1 \in \mathcal{C}(\mu_1, \mu_2)$. At the same time, by
(\ref{29}), (\ref{31}), and (\ref{32}) it follows that
\begin{equation}
  \label{39}
\nu_1 (I_\ell) \leq \left\{\begin{array}{ll}
\bar{\kappa} \gamma(\nu_0) + 2 \varDelta^2 K^{-1} \lambda (\nu_0),&  \qquad {\rm for} \  \ \ell\in {\sf V}_0  ; \\[.3cm]
 \gamma(\nu_0), &  \qquad {\rm for} \  \ \ell\in {\sf W}_0 .
\end{array} \right.
\end{equation}
Similarly, by (\ref{32}) -- (\ref{38}) we obtain
\begin{equation}
  \label{40}
\nu_1 (I_\ell H^{i}_{\ell'}) \leq \left\{\begin{array}{ll}
\varDelta \gamma(\nu_0) + \left[ \varDelta \bar{c} + 2 \varDelta^2 K^{-1} \right] \lambda(\nu_0), &  \
\ell, \ell' \in {\sf V}_0;\\[.3cm]
\varDelta \lambda (\nu_0), &  \
\ell \in {\sf V}_0,   \ell' \in {\sf W}_0 ;\\[.3cm]
\gamma(\nu_0) + \bar{c} \lambda(\nu_0), & \  \ell\in {\sf W}_0, \ell' \in {\sf V}_0;\\[.3cm]
\lambda (\nu_0), & \
\ell, \ell' \in {\sf W}_0.
\end{array} \right.
\end{equation}
These estimates complete the reconstruction over ${\sf V}_0$.
\vskip.1cm
\subsection{Reconstruction over ${\sf V}_j$: Proof of Lemma \ref{R3lm}}
\label{sec:2}
Here we assume that $\nu_j$ satisfies the following estimates, cf. (\ref{39}), where $A$ is as in (\ref{26a}):
\begin{equation}
  \label{41}
\nu_j (I_\ell) \leq \left\{\begin{array}{ll}
\left[\bar{\kappa} + A K^{-1}\right] \gamma(\nu_0) + 2 A K^{-1} \lambda (\nu_0),&  \qquad {\rm for} \  \ \ell\in {\sf U}_{j-1}  ; \\[.3cm]
 \gamma(\nu_0), &  \qquad {\rm for} \  \ \ell\in {\sf W}_{j-1} .
\end{array} \right.
\end{equation}
And also, cf. (\ref{40}),
\begin{equation}
  \label{42}
\nu_j (I_\ell H^{i}_{\ell'}) \leq \left\{\begin{array}{ll}
 \varDelta^j
 \gamma(\nu_0) + \bar{c}\varDelta^{j+1} \lambda(\nu_0), &  \
\ell, \ell' \in {\sf U}_{j-1};\\[.3cm]
\varDelta^j \lambda (\nu_0), &  \
\ell \in {\sf U}_{j-1},   \ell' \in {\sf W}_{j-1} ;\\[.3cm]
 j \gamma(\nu_0) + \bar{c} \lambda(\nu_0), & \  \ell\in {\sf W}_{j-1}, \ell' \in {\sf V}_{j-1};\\[.3cm]
\lambda (\nu_0), & \
\ell, \ell' \in {\sf W}_{j-1}.
\end{array} \right.
\end{equation}
Since ${\sf W}_{\varDelta-1} = \emptyset$, see (\ref{29}), for $j= \varDelta -1$ we have just the first lines in
(\ref{41}) and (\ref{42}), which yields (\ref{27}) and (\ref{28}), respectively, and thus the proof of Lemma \ref{R3lm}.
Note that (\ref{39}) agrees with  (\ref{41}) as $\varDelta^2 < A$, see (\ref{26a}). Also (\ref{40}) agrees with  (\ref{42}), which follows
from the fact that
\[
\bar{c} \varDelta + 2 \varDelta^2 K^{-1} < \bar{c} \varDelta + A K^{-1} \leq \bar{c} \varDelta + \bar{c}/2 < \bar{c} \varDelta^2
\leq \bar{c} \varDelta^{j+1}, \quad j = 1, \dots \chi-1,
\]
see (\ref{26a}) and (\ref{26b}).

Thus, our aim now is to prove that the estimates as in (\ref{41}) and (\ref{42}) hold also for $j+1$.
Note that the last lines in these estimates follow by claim (b) of Lemma \ref{Rlm}.
As above, we enumerate ${\sf V}_j = \{\ell_1 , \ell_2 , \cdots \}$ and  set
\[
\nu_j^{(n)} = R_{\ell_n}R_{\ell_{n-1}} \cdots R_{\ell_1}\nu_j.
\]
For $k\leq n$, by (\ref{21}) we have, cf. (\ref{32}),
\begin{eqnarray*}
\nu_j^{(n)} (I_{\ell_k}) & = & \nu_j^{(k)} (I_{\ell_k}) \leq
\sum_{\ell \in \partial \ell_k \cap {\sf U}_{j-1}}\kappa_{\ell_k\ell} \nu_j(I_\ell)
+ \sum_{\ell \in \partial \ell_k \cap {\sf W}_{j}}\kappa_{\ell_k\ell} \nu_j(I_\ell)\qquad \nonumber \\[.2cm]
& + & K^{-1} \sum_{i=1,2} \sum_{\ell, \ell' \in \partial \ell_k \cap {\sf U}_{j-1}} \nu_j (I_\ell H^i_{\ell'}) \nonumber \\[.2cm]
& + & K^{-1} \sum_{i=1,2} \sum_{\ell \in \partial \ell_k \cap {\sf U}_{j-1}} \sum_{  \ell' \in \partial \ell_k \cap {\sf W}_{j} } \nu_j (I_\ell H^i_{\ell'}) \nonumber \\[.2cm]
& + & K^{-1} \sum_{i=1,2} \sum_{\ell \in \partial \ell_k \cap {\sf W}_{j} }\sum_{   \ell' \in \partial \ell_k \cap {\sf U}_{j-1} } \nu_j (I_\ell H^i_{\ell'}) \nonumber \\[.2cm]
& + & K^{-1} \sum_{i=1,2} \sum_{\ell, \ell' \in \partial \ell_k \cap {\sf W}_{j}} \nu_j (I_\ell H^i_{\ell'}).
\end{eqnarray*}
Now we use the assumptions in (\ref{41}) and (\ref{42}) and obtain herefrom
\begin{eqnarray}
  \label{45}
 \nu_j^{(n)} (I_{\ell_k}) & \leq & \left[\bar{\kappa} + K^{-1} \left( \bar{\kappa} A + 2 \varDelta^j  \varDelta^2_j  + 2 j \varDelta_j
 \widetilde{\varDelta}_j   \right) \right]\gamma(\nu_0) \\[.2cm]
 & + &2 K^{-1} \left[\bar{\kappa} A + \bar{c} \varDelta^{j+1} \varDelta_j^2 +  \varDelta^j \varDelta_j\widetilde{ \varDelta}_j \right. \nonumber \\[.2cm]
 & + & \left.  \bar{c} \varDelta_j\widetilde{ \varDelta}_j +  \widetilde{ \varDelta}^2_j \right]\lambda (\nu_0) ,\nonumber
\end{eqnarray}
where
\[
\varDelta_j := | \partial \ell_k \cap {\sf U}_{j-1}| , \qquad \widetilde{\varDelta}_j:= | \partial \ell_k \cap {\sf W}_{j}|.
\]
To prove that
\[
 \bar{\kappa} A +  2 \varDelta^j  \varDelta^2_j  + 2 j \varDelta_j
 \widetilde{\varDelta}_j  \leq A
\]
see the first line in (\ref{41}),
we use (\ref{26a}), take into account that $\varDelta \geq 2$ (hence, $j \leq \varDelta^j$, $j=1, 2 , \dots \chi -1$) and obtain
\[
 2 \varDelta^j  \varDelta^2_j  + 2 j \varDelta_j
 \widetilde{\varDelta}_j \leq 2 \varDelta^j  \varDelta_j\left( \varDelta_j + \widetilde{\varDelta}_j (j/\varDelta^j)\right)\leq 2 \varDelta^{j+2}
 \leq A (1 - \bar{\kappa}),
\]
where we have taken into account that $j+2 \leq \chi+1$, see (\ref{26a}).
To prove that the coefficient at $\lambda (\nu_0)$ in (\ref{45}) agrees with that in (\ref{41}) we use the following
estimates
\begin{eqnarray*}
 & & \bar{c} \varDelta^{j+1} \varDelta_j^2 +  \varDelta^j \varDelta_j\widetilde{ \varDelta}_j +
 \bar{c} \varDelta_j\widetilde{ \varDelta}_j +  \widetilde{ \varDelta}^2_j \\[.2cm]
 & & \quad = \bar{c} \varDelta^{j+1} \varDelta_j \left( \varDelta_j + \widetilde{\varDelta}_j \varDelta^{-j}\right) +
 \varDelta^{j}\widetilde{\varDelta}_j \left( \varDelta_j + \widetilde{\varDelta}_j \varDelta^{-(j+1)}\right)\\[.2cm]
& & \quad  \leq \varDelta^2 + \varDelta^{j+2}  \leq 2\varDelta^{j+2}
 \leq A (1 - \bar{\kappa}).
\end{eqnarray*}
For $\ell \in {\sf U}_{j-1}$, $\nu_j^{(n)}(I_\ell) = \nu_j (I_\ell)$ and hence obeys the first line of (\ref{41}). For
$\ell \in {\sf W}_{j}$, again $\nu_j^{(n)}(I_\ell) = \nu_j (I_\ell)$ and hence obeys the second line of (\ref{41}).
Here we also used that $\bar{c} < 1/\varDelta^\chi$ and $j+1\leq \chi$, see (\ref{13}).
Thus,  (\ref{41}) with $j+1$ holds true.

Now we turn to estimating $\nu_j^{(n)} (I_{\ell} H^i_{\ell'})$. In the situation where  $\ell, \ell' \in {\sf U}_{j-1}\cup {\sf W}_j$, we have that
$\nu_j^{(n)} (I_{\ell} H^i_{\ell'}) = \nu_j (I_{\ell} H^i_{\ell'})$ and hence obeys (\ref{42}). Let us consider first the cases where
only one vertex of $\ell, \ell'$ lies in ${\sf V}_j$.

For $\ell' \in {\sf U}_{j-1}$ and $k \leq n$, by (\ref{23}) and the first and third lines in (\ref{42}) we obtain
\begin{eqnarray*}
 & & \nu_j^{(n)} (I_{\ell_k} H^i_{\ell'}) = \nu_j^{(k)} (I_{\ell_k} H^i_{\ell'})
 \leq \sum_{\ell \in \partial \ell_k \cap {\sf U}_{j-1}}\nu_j (I_{\ell} H^i_{\ell'})
 + \sum_{\ell \in \partial \ell_k \cap {\sf W}_{j}}\nu_j (I_{\ell} H^i_{\ell'}) \nonumber \\[.2cm]
 & & \quad \leq \left[\varDelta^j  \varDelta_j + j \widetilde{ \varDelta}_j \right] \gamma(\nu_0) +
 \left[\bar{c}\varDelta^{j+1} \varDelta_j + \bar{c} \widetilde{ \varDelta}_j \right] \lambda(\nu_0)\nonumber \\[.2cm]
 & & \quad \leq \varDelta^{j+1} \gamma(\nu_0) + \bar{c} \varDelta^{j+2}\lambda(\nu_0),
 \end{eqnarray*}
which yields the first line in (\ref{42}) with $j+1$.

For $\ell' \in {\sf W}_{j}$ and $k \leq n$, by (\ref{23}) and the second and fourth lines in (\ref{42}) it follows that
\begin{gather*}
 \nu_j^{(n)} (I_{\ell_k} H^i_{\ell'}) = \nu_j^{(k)} (I_{\ell_k} H^i_{\ell'}) \leq \sum_{\ell \in \partial \ell_k \cap {\sf U}_{j-1}}\nu_j (I_{\ell} H^i_{\ell'}) + \sum_{\ell \in \partial \ell_k \cap {\sf W}_{j}}\nu_j (I_{\ell} H^i_{\ell'}) \nonumber \\[.2cm]
 \leq
 \left(\varDelta^j \varDelta_j + \widetilde{ \varDelta}_j \right) \lambda(\nu_0) \leq \varDelta^{j+1} \lambda(\nu_0),
 \end{gather*}
 which agrees with the second line in (\ref{42}).

 For $\ell \in {\sf U}_{j-1}$ and $k \leq n$, by (\ref{22}) and the first and second lines in (\ref{42}) we get
\begin{eqnarray*}
 & & \nu_j^{(n)} (I_{\ell} H^i_{\ell_k}) = \nu_j^{(k)} (I_{\ell} H^i_{\ell_k}) \leq \nu_j (I_\ell) +
 \sum_{\ell' \in \partial \ell_k \cap {\sf U}_{j-1}}c_{\ell_k\ell'}\nu_j (I_{\ell} H^i_{\ell'}) \nonumber \\[.2cm]
 & & \quad + \sum_{\ell' \in \partial \ell_k \cap {\sf W}_{j}} c_{\ell_k\ell'}\nu_j (I_{\ell} H^i_{\ell'}) \leq
 \left[ \bar{\kappa} + A K^{-1} \right] \gamma(\nu_0)  \\[.2cm]
 & & \quad + 2A K^{-1} \lambda (\nu_0) +
 \left[\varDelta^j  \gamma(\nu_0) + \bar{c} \varDelta^{j+1}  \lambda (\nu_0) \right]
 \sum_{\ell' \in \partial \ell_k \cap {\sf U}_{j-1}}c_{\ell_k\ell'} \nonumber  \\[.2cm] & & \quad +
 \varDelta^j \lambda (\nu_0) \sum_{\ell' \in \partial \ell_k \cap {\sf W}_{j}} c_{\ell_k\ell'}. \nonumber
 \end{eqnarray*}
In order for this to agree with the first line in (\ref{42}), it is enough that the following holds
\begin{eqnarray}
 \label{48a}
 & &\qquad  \bar{\kappa} + AK^{-1} +  \varDelta^{j} \sum_{\ell' \in \partial \ell_k \cap {\sf U}_{j-1}}c_{\ell_k\ell'} \leq \varDelta^{j+1}, \\[.2cm]
& &  2 AK^{-1} +  \bar{c} \varDelta^{j+1} \sum_{\ell' \in \partial \ell_k \cap {\sf U}_{j-1}}c_{\ell_k\ell'} +
\varDelta^{j} \sum_{\ell' \in \partial \ell_k \cap {\sf W}_{j}}c_{\ell_k\ell'} \leq \bar{c} \varDelta^{j+2}. \nonumber
 \end{eqnarray}
Recall that we assume $\varDelta \geq 2$. By (\ref{26b}) and (\ref{13}) we get that the left-hand side of the first line in (\ref{48a}) does not exceed
\[
 \bar{\kappa}  + \bar{c}/2 + \varDelta^{-1} < 2 < \varDelta^{j+1}, \qquad {\rm for}  \ \  j=1, \dots , \chi-1.
\]
Likewise, the left-hand side of the second line in (\ref{48a}) does not exceed
\[
 \bar{c} + \bar{c} + \bar{c} \varDelta^{j} \leq \bar{c}( 2 + \varDelta^{j} ) < \bar{c} \varDelta^{j+2}  \qquad {\rm for}  \ \  j=1, \dots , \chi-1.
\]
For $\ell \in {\sf W}_{j}$ and $k \leq n$, by (\ref{22}) and the third and fourth lines in (\ref{42}) we get
\begin{eqnarray}
\label{49}
& &  \nu_j^{(n)} (I_{\ell} H^i_{\ell_k}) = \nu_j^{(k)} (I_{\ell} H^i_{\ell_k})
 \leq \nu_j (I_\ell)  \nonumber \\[.2cm] & &  + \sum_{\ell' \in \partial \ell_k \cap {\sf U}_{j-1}}c_{\ell_k\ell'}\nu_j (I_{\ell} H^i_{\ell'})
 + \sum_{\ell' \in \partial \ell_k \cap {\sf W}_{j}} c_{\ell_k\ell'}\nu_j (I_{\ell} H^i_{\ell'}) \\[.2cm]
 & & \leq  \gamma(\nu_0) + \left[j \gamma(\nu_0) +
 \bar{c} \lambda (\nu_0)\right]  \sum_{\ell' \in \partial \ell_k \cap {\sf U}_{j-1}}c_{\ell_k\ell'}
 +  \lambda (\nu_0) \sum_{\ell' \in \partial \ell_k \cap {\sf W}_{j}} c_{\ell_k\ell'} \nonumber \\[.2cm]
 & & \leq (1 + j \bar{c}) \gamma(\nu_0) + \left( \bar{c}\sum_{\ell' \in \partial \ell_k \cap {\sf U}_{j-1}}c_{\ell_k\ell'} +
 \sum_{\ell' \in \partial \ell_k \cap {\sf W}_{j}}c_{\ell_k\ell'} \right) \lambda(\nu_0), \nonumber
 \end{eqnarray}
 which clearly agrees with the third line in (\ref{42}).

Now we consider the cases where both $\ell, \ell'$ lie in ${\sf V}_j$. For $k< m\leq n$, by first (\ref{22}) and (\ref{23}), and then by (\ref{21}),  we have
\begin{eqnarray}
  \label{50}
 \nu_j^{(n)} (I_{\ell_k} H^i_{\ell_m}) & = & \nu_j^{(m)} (I_{\ell_k} H^i_{\ell_m}) \leq  \nu_j^{(k)} (I_{\ell_k}) + \sum_{\ell'\in \partial \ell_m}
 c_{\ell_m\ell'} \nu_j^{(k)}(I_{\ell_k}H^i_{\ell'}) \nonumber \\[.2cm]
& \leq &\sum_{\ell \in \partial \ell_k} \kappa_{\ell_k\ell}\nu_j (I_\ell) +
K^{-1} \sum_{s=1,2} \sum_{\ell, \ell'\in \partial \ell_k} \nu_j (I_\ell H^s_{\ell'}) \nonumber \\[.2cm]
& + & \sum_{\ell'\in \partial \ell_m}
 c_{\ell_m\ell'} \sum_{\ell\in \partial \ell_k} \nu_j(I_{\ell}H^i_{\ell'}).
\end{eqnarray}
The next step is to split the sums in (\ref{50}) as it has been done in, e.g., (\ref{49}), and then use (\ref{41}) and (\ref{42}).
By doing so we get
\begin{eqnarray*}
 & & \nu_j^{(n)} (I_{\ell_k} H^i_{\ell_m})
 \leq  \left[ (\bar{\kappa} + A K^{-1}) \gamma (\nu_0) + 2A  K^{-1} \lambda (\nu_0)\right]
 \sum_{\ell \in \partial \ell_k \cap {\sf U}_{j-1}} \kappa_{\ell_k\ell} \nonumber \\[.2cm]
 & & \quad + \gamma(\nu_0) \sum_{\ell \in \partial \ell_k \cap {\sf W}_{j}} \kappa_{\ell_k\ell}   +
 2 K^{-1} \varDelta_j^2 \left[\varDelta^j  \gamma(\nu_0) + \bar{c}\varDelta^{j+1} \lambda (\nu_0) \right] \nonumber \\[.2cm]
 & & \quad + 2K^{-1} \varDelta_j \widetilde{\varDelta}_j \left[ \varDelta^j \lambda (\nu_0) + j \gamma(\nu_0) + \bar{c}\lambda( \nu_0) \right]
 + 2 K^{-1} \widetilde{\varDelta}_j^2 \lambda (\nu_0)
  \nonumber \\[.2cm]
 & & \quad + \varDelta_j \left[\varDelta^j \gamma(\nu_0) +
 \bar{c}\varDelta^{j+1} \lambda (\nu_0) \right] \sum_{\ell' \in \partial \ell_m\cap {\sf U}_{j-1}} c_{\ell_m\ell'}   \nonumber \\[.2cm]
& & \quad +
 \varDelta^j \varDelta_j \lambda (\nu_0) \sum_{\ell' \in \partial \ell_m\cap {\sf W}_{j}} c_{\ell_m\ell'} +
 \widetilde{\varDelta}_j (j \gamma(\nu_0) + \bar{c} \lambda (\nu_0) ) \sum_{\ell' \in \partial \ell_m\cap {\sf U}_{j-1}}
 c_{\ell_m\ell'}  \nonumber \\[.2cm]
& & \quad + \widetilde{\varDelta}_j  \lambda (\nu_0) \sum_{\ell' \in \partial \ell_m\cap {\sf W}_{j}} c_{\ell_m\ell'}.
   \end{eqnarray*}
In order for this to agree with the first line in (\ref{42}), it is enough that the following two estimate hold
\begin{eqnarray}
 \label{51a}
& & (\bar{\kappa} + A K^{-1}) \sum_{\ell \in \partial \ell_k \cap {\sf U}_{j-1}} \kappa_{\ell_k\ell}
+ \sum_{\ell \in \partial \ell_k \cap {\sf W}_{j}} \kappa_{\ell_k\ell}
+ 2K^{-1} \varDelta_j^2 \varDelta^j \\[.2cm]
& & \quad +
2K^{-1} j \varDelta_j \widetilde{ \varDelta}_j  + \varDelta_j \varDelta^j \sum_{\ell' \in \partial \ell_m\cap {\sf U}_{j-1}} c_{\ell_m\ell'} +
\widetilde{ \varDelta}_j \sum_{\ell' \in \partial \ell_m\cap {\sf U}_{j-1}} c_{\ell_m\ell'} \nonumber \\[.2cm] & & \quad \leq \varDelta^{j+1}\nonumber \\[.3cm]
\label{51b}
& & 2 A K^{-1} \sum_{\ell \in \partial \ell_k \cap {\sf U}_{j-1}} \kappa_{\ell_k\ell} + 2K^{-1} \varDelta_j^2 \bar{c} \varDelta^{j+1} +
2K^{-1} \varDelta_j \widetilde{\varDelta}_j \varDelta^{j} \\[.2cm] &  & \quad  +
2K^{-1} \bar{c} \varDelta_j \widetilde{\varDelta}_j + 2K^{-1}  \widetilde{\varDelta}_j^2 + \bar{c} \varDelta_j \varDelta^{j+1}
\sum_{\ell' \in \partial \ell_m\cap {\sf U}_{j-1}} c_{\ell_m\ell'} \nonumber \\[.2cm]
& & \quad + \varDelta_j \varDelta^{j} \sum_{\ell' \in \partial \ell_m\cap {\sf W}_{j}} c_{\ell_m\ell'} +  \bar{c} \widetilde{\varDelta}_j
\sum_{\ell' \in \partial \ell_m\cap {\sf U}_{j-1}} c_{\ell_m\ell'} \nonumber \\[.2cm]
& & + \quad \widetilde{\varDelta}_j
\sum_{\ell' \in \partial \ell_m\cap {\sf W}_{j}} c_{\ell_m\ell'} \leq \bar{c} \varDelta^{j+2}.  \nonumber
\end{eqnarray}
Taking into account that $\bar{\kappa} < 1$ and (\ref{26b}) one can show that the left-hand side of (\ref{51a}) does not exceed
\begin{gather*}
1 + \bar{c}/2 + 2 K^{-1} \varDelta^j \varDelta_j \left( \varDelta_j + \widetilde{\varDelta}_j (j/\varDelta^j)\right) + \bar{c} \varDelta^{j+1}
  \\[.2cm]\leq 1 + \bar{c}/2 + \bar{c}/2 + \bar{c} \varDelta^{j+1} < 2 + \frac{1}{\varDelta^{\chi}} < \varDelta^{j+1}.
\end{gather*}
To prove (\ref{51b}) we use (\ref{26b}), (\ref{13}), the inequality $\varDelta_j \widetilde{\varDelta}_j \leq \varDelta^2/4$, and perform
the following calculations
\begin{gather*}
{\rm LHS(\ref{51b})} \leq 2AK^{-1} \bar{\kappa} + \frac{1}{2} K^{-1}  \varDelta^{j+2} + 2K^{-1} \left(\varDelta_j^2 +
\bar{c} \varDelta_j \widetilde{\varDelta}_j + \widetilde{\varDelta}^2_j\right)  \\[.2cm]
+ \varDelta_j \sum_{\ell' \in \partial \ell_m\cap {\sf U}_{j-1}} c_{\ell_m\ell'} +
\varDelta^j \varDelta_j \sum_{\ell' \in \partial \ell_m\cap {\sf W}_{j}} c_{\ell_m\ell'} \\[.2cm]
+ \bar{c} \widetilde{\varDelta}_j \sum_{\ell' \in \partial \ell_m\cap {\sf U}_{j-1}} c_{\ell_m\ell'} +
\widetilde{\varDelta}_j \sum_{\ell' \in \partial \ell_m\cap {\sf W}_{j}} c_{\ell_m\ell'}\\[.2cm]
\leq \bar{c} + \frac{\bar{c} \varDelta^{j+2}}{8 \varDelta^{\chi+1}} + \frac{\bar{c} \varDelta^{2}}{2 \varDelta^{\chi+1}} + \bar{c} \varDelta^{j+1}
+ \bar{c} \varDelta < \bar{c} \varDelta^{j+2},
\end{gather*}
which holds even for $j=1$, $\chi = 2$, and $\varDelta = 2$.

Next, for $k\leq n$,  by (\ref{24}) we have
\begin{equation}
  \label{52}
\nu_j^{(n)} (I_{\ell_k} H^i_{\ell_k})= \nu_j^{(k)} (I_{\ell_k} H^i_{\ell_k}) \leq
\sum_{\ell\in \partial \ell_k} \nu_j (I_\ell) + \sum_{\ell, \ell'\in \partial \ell_k} c_{\ell_k \ell'} \nu_j(I_\ell H^i_{\ell'})
\end{equation}
As above, we split the sums in (\ref{52}) and then use (\ref{41}) and (\ref{42}), and obtain
\begin{eqnarray}
  \label{53}
& & \nu_j^{(n)} (I_{\ell_k} H^i_{\ell_k}) \leq \varDelta_j \left[\bar{\kappa} + A K^{-1} \right] \gamma(\nu_0) +
\varDelta_j 2A K^{-1} \lambda (\nu_0) \nonumber\\[.2cm] & &  \quad
+ \widetilde{\varDelta}_j \gamma(\nu_0) + \varDelta_j \left[ \varDelta^j \gamma(\nu_0) +
\bar{c} \varDelta^{j+1} \lambda(\nu_0) \right]\sum_{\ell' \in \partial \ell_k \cap {\sf U}_{j-1}} c_{\ell_k \ell'} \nonumber \\[.2cm]
& &  \quad
+ \varDelta_j \varDelta^j\lambda (\nu_0) \sum_{\ell' \in \partial \ell_k \cap {\sf W}_{j}} c_{\ell_k \ell'} +
\widetilde{\varDelta}_j (j \gamma(\nu_0)
+ \bar{c} \lambda (\nu_0)) \sum_{\ell' \in \partial \ell_k \cap {\sf U}_{j-1}} c_{\ell_k \ell'} \nonumber \\[.2cm]
& &  \quad
+ \widetilde{\varDelta}_j  \lambda(\nu_0) \sum_{\ell' \in \partial \ell_k \cap {\sf W}_{j}} c_{\ell_k \ell'}.
\end{eqnarray}
In order for this to agree with the first line in (\ref{42}), it is sufficient that the following two inequalities hold
\begin{eqnarray}
 \label{53a}
 & & \varDelta_j \left[\bar{\kappa} + A K^{-1} \right] + \widetilde{\varDelta}_j + \left( \varDelta^j  \varDelta_j + j \widetilde{\varDelta}^j
 \right) \sum_{\ell' \in \partial \ell_k \cap {\sf U}_{j-1}} c_{\ell_k \ell'} \leq \varDelta^{j+1}, \qquad \quad
 \\[.3cm]
 \label{53b}
 & & 2 A K^{-1} \varDelta_j + \bar{c} \varDelta^{j+1}  \varDelta_j \sum_{\ell' \in \partial \ell_k \cap {\sf U}_{j-1}} c_{\ell_k \ell'}
 + \varDelta^{j}  \varDelta_j \sum_{\ell' \in \partial \ell_k \cap {\sf W}_{j}} c_{\ell_k \ell'}\\[.2cm]
& & \quad  +  \bar{c} \widetilde{\varDelta_j}  \sum_{\ell' \in \partial \ell_k \cap {\sf U}_{j-1}} c_{\ell_k \ell'} +
\widetilde{\varDelta}_j  \sum_{\ell' \in \partial \ell_k \cap {\sf W}_{j}} c_{\ell_k \ell'} \leq \bar{c} \varDelta^{j+2}. \nonumber
\end{eqnarray}
By means of (\ref{26b}) we get
\begin{gather*}
 {\rm LHS(\ref{53a})} \leq \varDelta + \varDelta AK^{-1} + \bar{c} \varDelta^{j+1} < \varDelta + \frac{1}{2 \varDelta^{\chi-1}} + 1 < \varDelta^{j+1}.
\end{gather*}
Similarly,
\begin{gather*}
 {\rm LHS(\ref{53b})} \leq \bar{c} \varDelta_j + \varDelta_j \sum_{\ell' \in \partial \ell_k \cap {\sf U}_{j-1}} c_{\ell_k \ell'} +
\varDelta^{j}  \varDelta_j \sum_{\ell' \in \partial \ell_k \cap {\sf W}_{j}} c_{\ell_k \ell'} \\[.2cm]
 + \bar{c}\widetilde{\varDelta}_j \sum_{\ell' \in \partial \ell_k \cap {\sf U}_{j-1}} c_{\ell_k \ell'}
+\widetilde{\varDelta}_j \sum_{\ell' \in \partial \ell_k \cap {\sf W}_{j}} c_{\ell_k \ell'} \\[.2cm]
\leq   \bar{c} \varDelta + \bar{c} \varDelta^j \varDelta_j + \bar{c} \widetilde{\varDelta}_j < \bar{c} \varDelta + \bar{c} \varDelta^{j+1} \leq
\bar{c} \varDelta^{j+2}.
\end{gather*}
Now we consider the case where $ m <k \leq n$. By (\ref{23}), and then by (\ref{22}), we have
\begin{eqnarray}
  \label{54}
 \nu_j^{(n)} (I_{\ell_k} H^i_{\ell_m})& = & \nu_j^{(k)} (I_{\ell_k} H^i_{\ell_m}) \leq \sum_{\ell \in \partial \ell_k} \nu_j^{(m)} (I_\ell H^i_{\ell_m})   \\[.2cm]
 & \leq &  \sum_{\ell \in \partial \ell_k} \nu_j(I_\ell) + \sum_{\ell \in \partial \ell_k} \sum_{\ell' \in \partial \ell_m}c_{\ell_m\ell'}
 \nu_j (I_{\ell} H^i_{\ell'}). \nonumber
\end{eqnarray}
Again we split the sums in (\ref{54}) and then use (\ref{41}) and (\ref{42}), and obtain that
\[
 \nu_j^{(n)} (I_{\ell_k} H^i_{\ell_m}) \leq {\rm RHS}(\ref{53}).
\]
Thus, we have that (\ref{42}) with $j+1$ holds also in this case. The proof is complete.
\subsection{The proof of Lemma \ref{dclm}}
Assume that we have given $\nu^x_{s-1}\in \mathcal{C}(\mu^x, \mu)$ with the properties in question. Then we split ${\sf D}_{N-s-1}$ into
independent subsets by taking intersections with the sets ${\sf V}_j$, as in (\ref{3}). Let $\ell^1 , \dots , \ell^m$ be a numbering of
${\sf D}_{N-s-1}\cap {\sf V}_0$. Set
\[
\tilde{\nu}^x_0 = \nu^x_{s-1} \quad {\rm and} \quad \tilde{\nu}^x_k = R_{\ell^k}  \tilde{\nu}^x_{k-1}, \qquad k=1, \dots , m,
\]
where $R_\ell$ is defined in (\ref{16}). Thus, $\tilde{\nu}_m^x$ is $\mathcal{F}_{{\sf D}_{N-1}}$-measurable, and
$\tilde{\nu}_m^x (I_\ell)$ and $\tilde{\nu}_m^x (I_\ell H^i_{\ell'})$, $\ell, \ell' \in {\sf D}_{N-s-1}$, satisfy
the inequalities in (\ref{39}) and (\ref{40}), respectively, in which  the right-hand sides contain $\gamma_{{\sf D}_{N-s}} (\nu^x_{s-1})$ and
$\lambda_{{\sf D}_{N-s}} (\nu^x_{s-1})$. Then we perform the reconstruction over the remaining independent subsets of
${\sf D}_{N-s-1}$ and obtain an element of $\mathcal{C}(\mu^x, \mu)$, which we denote by $\nu^x_s$. Its $\mathcal{F}_{{\sf D}_{N-1}}$-measurability is then guarateed
by construction, and the parameters $\gamma_{{\sf D}_{N-s-1}} (\nu^x_{s})$ and
$\lambda_{{\sf D}_{N-s-1}} (\nu^x_{s})$ satisfy the first-line inequalities in (\ref{41}) and (\ref{42}), respectively, and hence (\ref{dc3})
with $\gamma_{{\sf D}_{N-s}} (\nu^x_{s-1})$ and
$\lambda_{{\sf D}_{N-s}} (\nu^x_{s-1})$ on the right-hand side. The $\mathcal{F}_{{\sf D}_{N-1}}$-measurability of $\nu_0^x$ is straightforward.

\end{document}